\documentclass[a4, 11pt]{article}
\usepackage{amsmath}
\usepackage{amsfonts}
\usepackage{mathrsfs}
\usepackage{dsfont}
\usepackage{color}
\title{Upper Bound for Large Deviations of Reversible Diffusion Processes}
\author{Ann-Kathrin Jarecki}
\date{}
\newtheorem{thm}{Theorem}
\newtheorem{cor}[thm]{Corollary}
\newtheorem{lemma}[thm]{Lemma}

\newtheorem{defn}[thm]{Definition}
\newtheorem{rem}[thm]{Remark}

\numberwithin{thm}{section} \numberwithin{equation}{section}

\newenvironment{proof}{{\bf
Proof:\,}}{\hspace*{\fill}\rule{1.2ex}{1.2ex}\\ }

\newcommand{\D}{\mathds{D}}
\newcommand{\E}{\mathcal{E}}
\newcommand{\X}{\mathcal{X}}
\newcommand{\Y}{\mathcal{Y}}
\newcommand{\C}{\mathcal{C}}
\newcommand{\A}{\mathcal{A}}
\newcommand{\N}{\mathds{N}}
\newcommand{\R}{\mathds{R}}
\newcommand{\Pp}{\mathbb{P}}
\newcommand{\Ee}{\mathbb{E}}
\newcommand{\Prob}{\mathds{P}}
\newcommand{\eins}{\mathds{1}}

\DeclareMathOperator{\diam}{diam}
\DeclareMathOperator*{\esss}{esssup}
\DeclareMathOperator*{\essi}{essinf}

\begin{document}

\maketitle

\begin{abstract}
For a Markov process associated with a diffusion type Dirichlet form
an upper bound is shown for the law of the finite dimensional
distributions of the process. Under some more assumptions on the
underlaying space this is also shown for the law of the Markov
process itself. In the last section we want to give an application
to the Wasserstein diffusion.
\end{abstract}

\section{Introduction}
Let $(\X,\mathcal{F},m)$ be a probability space and $L^p=L^p(\X;m)$,
$p \in [1,\infty]$, the corresponding $L^p$--space with norm $\|
\cdot \|_{L^p}$ and inner product $(\cdot,\cdot)_{L^p}$. The
underlaying topological space $\X$ is assumed to be polish.

We start with a Dirichlet Form $\E$ on $\D \subset L^2$, i.e. $\D$
is a close subset of $L^2$ and $\E$ is a positive semidefinite,
symmetric and closed bilinear form with the property, that, if $f
\in \D$, then also $f \wedge 1 \in \D$ and $\E(f \wedge 1,f \wedge
1)\leq \E(f,f)$. Further we assume that the Dirichlet form is
conservative and local. So $\E$ is of diffusion type, i.e $\E$ has
no killing nor jumping measure. Hence, there is a Markov process
$(X_t)_{t\geq 0}$ associated with $\E$ with continuous trajectories.
The associated Markov semi--group we denote by $\{T_t\}$.

We define $\D_b= \D \cap L^{\infty}$ and the functional $I: \D_b
\times \D_b \times \D_b \to \R$ by

\begin{equation} \label{Ifunctional}
(f,g;h) \mapsto I(f,g;h)=\E(gh,f) + \E(fh,g) -\E(fg,h),
\end{equation}

with the  convention that $I_f(h)=I(f,f;h)$. The subset $\D_0$ of
$\D_b$ is given through

\begin{equation}\label{Dnull}
\D_0=\{f \in \D_b \ :\ I_f(h)\leq \|h\|_{L^1} \mbox{ for all } h \in
\D_b\}.
\end{equation}

Herewith we can define the intrinsic metric $d$:

\begin{equation}\label{iM}
d(A,B)=\sup_{f \in \D_0} \big\{ \essi_{x \in B} f(x)-\esss_{y \in A}
f(y)\big\},
\end{equation}

for two measurable subsets $A$ and $B$ of $\X$. We put $\sup
\emptyset = -\infty$ and $\inf \emptyset = \infty$.

Under this conditions, but without assuming $\X$ to be polish, Hino
and Ram\'{\i}rez showed in \cite{HR} that for all measurable $A$ and
$B$

\begin{equation}\label{HR}
\lim_{t \to 0} t \log P_t(A,B)=-\frac{d(A,B)^2}{2} \ ,
\end{equation}
where $P_t(A,B)=\int_A T_t \eins_B d m$.

This result was the basis of our considerations. Now, we have to
introduce some more notations to formulate our result precisely.

Let $X : [0,1] \times \Omega \to \X$, be the Markov process with
values in $\X$ associated with the Dirichlet form $\E$ on a
probability space $(\Omega, \Pp)$. For simplicity, we always assume
$\Omega = \mathcal{C}([0,1],\X)$ and $X_t(\omega)=\omega(t)$. For $s
>0$, we consider the time--scaled process

$$X^s_t= X_{s\cdot t},\quad t \in [0,1].$$

Let $\Pp^s$ be the probability measure on $\Omega$ induced by
$X^s_\cdot$. For an arbitrary partition $\Delta=\Delta^n=\{0=
t_0<t_1<\ldots<t_{n-1}<t_n=1\}$ of the unit interval $[0,1]$ we
define the projection $\Pi_\Delta : \C([0,1],\X) \to \X^{n+1}$ on
the values at $t_0,t_1,\ldots,t_n$:

$$\Pi_\Delta (X_\cdot)= (X_{t_0},X_{t_1},\ldots,X_{t_n}).$$

Further for a $(n+1)$--tuple $\A=(A_0,A_1,\ldots,A_{n+1}) \in
\X^{n+1}$ we define

$$\A_\Delta=\Pi_\Delta^{-1}(\A)=\Big\{\gamma : [0,1] \to \X
\mbox{ continuous curve } \mid \Pi_\Delta(\gamma) \in \A \Big\}.$$

First of all, in this work it is our aim to derive an upper bound
for the finite dimensional distributions of the Markov process
$X_t$, i.e.

\begin{eqnarray} \label{1}
\Prob(\{\omega : X_{\cdot}(\omega)\in \mathcal{A}_{\Delta}\})&=& \int_{A_0}\int_{A_1}
\ldots \int_{A_n}T_{t_n - t_{n-1}}\ldots T_{t_2
- t_{1}}T_{t_1 - t_{0}} dm\\
&\leq& \sqrt{m(A_0)}\sqrt{m(A_n)} \exp \big( -
\sum_{i=0}^{n-1}\frac{d^2(A_i,A_{i+1})}{t_{i+1}-t_i}\big).
\end{eqnarray}

Therefor we frequently apply the 'Integrated Gaussian estimates`,
also known as the 'Method of Davies`.

Accordingly we derive the upper bound of a short--time asymptotic
for the law of the Markov process. For this we use a version of the theorem of
Dawson--G\"artner to lift up the upper bound of the finite
dimensional distributions to an upper bound of the Markov process
itself.

To be more precise, if  $H :
\mathcal{C}([0,1],\X) \to [0,\infty]$ is defined by
$$H(\gamma)=\sup_{\Delta^n}H_{\Delta^n}(\gamma)=\frac{1}{2}
\sum_{i=0}^{n-1}
\frac{d(\gamma(t_i),\gamma(t_{i+1}))^2}{t_{i+1}-t_i},$$ where the
supremum is taken over all partitions $\Delta^n$ of the unit
interval, then we will show that $H$ coincides with the energy $\widetilde{H}$ of a curve
\begin{equation}
\widetilde{H}(\gamma):=\left\{ \begin{array}{cl}
                            \frac{1}{2}\int_0^1|\dot{\gamma}|^2(r)\mathrm{d}r & \mathrm{,\,if } \gamma \in AC^2([0,1],\X)\\
                            \infty & \mathrm{,\,else.}
                            \end{array}\right.
\end{equation}
Finally we prove the following main theorem:

\begin{thm}
For all $\alpha > 0$ and for all compact subsets $\Gamma$ of $\{\gamma : \widetilde{H}(\gamma)> \alpha\}$  we have
$$\lim_{s \to 0}s \log \Pp^s(\{\omega : X_\cdot(\omega) \in \Gamma\}) \leq -\alpha.$$
\end{thm}

\section{The Intrinsic Metric}

We recall the notation

$$\D_0=\{f \in \D_b \ :\ I_f(h)\leq \|h\|_{L^1} \mbox{ for all } h
\in \D_b=\D \cap L^{\infty}\}.$$

With this, we define in an intrinsic way a pseudo metric $d$ on $\X$
by

\begin{equation}\label{pseudoMet}
d(x,y)=\sup_{f \in \D_0} \big\{f(x)-f(y)\big\}.
\end{equation}

In general, $d$ may be degenerate, i.e. $d(x,y)=\infty$ or
$d(x,y)=0$ for some $x\neq y$. If we make the assumption\\[-1ex]

{\bf (A)} The topology induced by $d$ is equivalent to the original
topology. \\[-1ex]

then the following properties are equivalent

\begin{itemize}
\item $d$ is non--degenerated
\item $d(x,y)<\infty$ for all $x,y \in \X$
\item $\X$ is connected.
\end{itemize}

For two measurable sets $A$ and $B$ the intrinsic metric is given by

\begin{equation} \label{intrinsicmetric}
d(A,B)=\sup_{f \in \D_0}\left\{\essi_{x \in B} f(x) -\esss_{y \in A}
f(y)\right\},
\end{equation}

As is customary we take $\sup \emptyset = -\infty$ and $\inf
\emptyset = \infty$.

The following theorem can be found in \cite{HR} (Theorem 1.2).

\begin{thm} \label{intrmetrik}
Let $A$ be a positive measure set; then there exists an  (a.e.)
unique $[0,\infty]$--valued measurable function $d_A$ such that

\begin{itemize}
\item $d_A \wedge N \in \D_0$ for any $N \geq 0$.
\item $d_A = 0$ a.e. on $A$.
\item $d_A$ is the (a.e.) largest function that is satisfies the two previous requirements.
\end{itemize}
Moreover, if $B$ is another measurable set, then
$$d(A,B)=\essi_{x \in B} d_A(x).$$

\end{thm}

For further details of the intrinsic metric, see also for example
\cite{sturmI}.

\section{Upper Bound for Finite Dimensional Distributions}

We carry over the notation from the introduction. That is
$(\X,\mathcal{F},m)$ is a probability space and $L^p=L^p(\X;m)$, $p
\in [1,\infty]$, the corresponding $L^p$--space with norm $\| \cdot
\|_{L^p}$ and inner product $(\cdot,\cdot)_{L^p}$. $\E$ is a
Dirichlet form on $\D \subset L^2$, which we assume to be
conservative and local. The goal of this section is to derive the
following theorem, which gives us an upper bound for the finite
dimensional distributions of the Markov process $X_t$ associated to
our Dirichlet form $\E$ on the probability space $(\Omega,\Pp)$, $X
: \Omega \to \mathcal{C}([0,1],\X)$.

\begin{thm} \label{thm}
For a partition $\Delta=\{0=t_0 < t_1 < \ldots < t_n=1\}$ and for
all $\A=(A_0,A_1,\ldots,A_{n}) \in \X^{n+1}$ define
$\A_\Delta=\Pi_\Delta^{-1}(\A)=\Big\{\gamma : [0,1] \to \X \mbox{
continuous } \mid \Pi_\Delta(\gamma) \in \A \Big\}.$ Then
\begin{eqnarray*}
\Prob(X_{\cdot}\in \mathcal{A}_{\Delta})&=&
\int_{A_0}\int_{A_1} \ldots \int_{A_n}T_{t_n - t_{n-1}}\ldots T_{t_2
- t_{1}}T_{t_1 - t_{0}} dm\\
&\leq& \sqrt{m(A_0)}\sqrt{m(A_n)} \exp \big( - \frac{1}{2}
\sum_{i=0}^{n-1}\frac{d^2(A_i,A_{i+1})}{t_{i+1}-t_i}\big).
\end{eqnarray*}
\end{thm}

We need the following lemmata:

\begin{lemma} \label{argh}
Let $\omega, u \in \D$. Then for all fix $\alpha \in \mathds{R}$
\begin{eqnarray*}
-2\E(e^{2 \alpha \omega}u,u) &\leq & 2 \alpha^2 \E(\omega u^2 e^{2 \alpha \omega}, \omega)
- \alpha^2 \E(\omega^2,u^2 e^{2 \alpha \omega})\\[1.5ex]
&=& \alpha^2 I_{\omega}(u^2 e^{2 \alpha \omega}),
\end{eqnarray*}
here $I_f(h)=I(f,f;h)=2 \E(fh,f) -\E(fg,h)$ (cf.
(\ref{Ifunctional})).
\end{lemma}

\begin{proof}
We put $f=(u,\omega) \in \D^2$ then according to results of
\cite{BH91} we get the following estimates

\begin{itemize}
\item[(1)] $\E(u,e^{2 \alpha \omega}u)= 2 \int e^{2 \alpha \omega}
d\sigma_{1,1}^f + 4 \alpha \int u e^{2 \alpha \omega} d\sigma_{1,2}^f$

\item[(2)] $\E(\omega u^2 e^{2 \alpha \omega},\omega)=2 \int u \omega
e^{2 \alpha \omega}d\sigma_{1,2}^f + \int u^2 e^{2 \alpha \omega}d\sigma_{2,2}^f
+ 2 \alpha \int \omega u^2 e^{2 \alpha \omega}d\sigma_{2,2}^f$

\item[(3)] $\E(\omega^2,u^2 e^{2 \alpha \omega})=4 \int \omega
u e^{2 \alpha \omega} d\sigma_{2,1}^f + 4\alpha \int \omega u^2
e^{2 \alpha \omega}d\sigma_{2,2}^f$
\end{itemize}

Therefore we get
\begin{eqnarray*}
\lefteqn{ 2 \E(u,e^{2 \alpha \omega}u) + 2 \alpha^2 \E(\omega u^2
e^{2 \alpha \omega},\omega) - \alpha^2 \E(\omega^2,u^2 e^{2 \alpha \omega})
\qquad {}}\\[1.5ex]
&&= 2 \int e^{2 \alpha \omega} d\sigma_{1,1}^f + 4 \alpha \int u e^{2 \alpha \omega}
d\sigma_{1,2}^f + 2 \alpha^2 \int u^2 e^{2 \alpha \omega}d\sigma_{2,2}^f\\[1.5ex]
&&= \E(e^{2 \alpha \omega}u,2 e^{2 \alpha \omega}u)\\
&&\geq 0
\end{eqnarray*}
and we obtain the claim.
\end{proof}

The next lemma is a version of the 'Integrated Gaussian estimates`,
also known as the 'The Method of Davies`, see for example
\cite{SII}.

\begin{lemma} \label{GauEst}
Let $\omega \in \D_0$ and $u_t = T_t f$, where $T_t$ is the
Markovian semi--group associated with the Dirichlet form $\E$. Then
for all $t
> 0$ and fixed, but arbitrary $\alpha \in \mathds{R}$
$$\|e^{\alpha \omega}u_t\|_{L^2}\leq e^{\alpha^2 t/2} \|e^{\alpha \omega}u_0\|_{L^2}.$$
\end{lemma}

\begin{proof}
\begin{eqnarray*}
\|e^{\alpha \omega}u_t\|^2_{L^2}-\|e^{\alpha \omega}u_0\|^2_{L^2}&=&\int_\X e^{2\alpha \omega}(u_t^2-u_0^2)dm\\
&= &2\int_0^t\int_\X (\frac{\partial}{\partial r} u_r) u_r e^{2\alpha \omega}dm\,dr\\
&= &2\int_0^t - \lim_{t \to 0} \int_\X T_r (\frac{f - T_t f}{t}) T_r f \, e^{2\alpha \omega}dm\,dr\\
&= &-2\int_0^t \E(u_r,e^{2\alpha \omega}u_r)dr\\
&\leq& \int_0^t \alpha^2 I_{\omega}(u_r^2 e^{2 \alpha \omega})dr\\
&\leq& \alpha^2 \int_0^t \|e^{2 \alpha \omega}u_r^2\|_{L^1}dr\\
&=& \alpha^2 \int_0^t \|e^{\alpha \omega}u_r\|_{L^2}^2dr.
\end{eqnarray*}

From this, the claim follows by Gronwall's Lemma.
\end{proof}

Before proving the theorem in its full generality we want to show
the special case of two dimensional distributions. This is also
shown in the paper by Hino and Ram\'{\i}rez \cite{HR} in a similar
way.

\begin{cor} \label{korollar}
Let $A$ and $B$ two measurable subsets of $\X$, $s > 0$, then

\begin{eqnarray*} \Pp(X_0 \in A,X_s \in B) &= & \Pp^s(X_0\in A,X_1\in B) \\[1.5ex]
&\leq& \sqrt{m(A)}\sqrt{m(B)}\ e^{-\frac{d^2(A,B)}{2s}}.
\end{eqnarray*}

\end{cor}

\begin{proof}
Suppose that $d(A,B)<\infty$. Let $\omega=d_A \wedge d(A,B)$. By
definition (\ref{intrinsicmetric}) $d(A,B)$ lies in $\D_0$ and $d_A$
as well by theorem (\ref{intrmetrik}). Hence we see $\omega \in
\D_0$. Let $\psi=\alpha \omega$ for an arbitrary fix $\alpha \in
\mathds{R}$ and $u_s=T_s \mathds{1}_A$. With this the assumptions of
the previous lemma are fulfilled and we get

\begin{equation} \label{gleichung1}
\|e^{\alpha \omega}T_s \mathds{1}_A\|_{L^2}\leq e^{\alpha^2
s/2}\|e^{\alpha \omega}\mathds{1}_A\|_{L^2} \stackrel{\omega\equiv
0\, on\, A}{=}e^{\alpha^2 s/2}\sqrt{m(A)}.
\end{equation}

Accordingly, for $v_s=T_s \mathds{1}_B$ we get

\begin{equation}\label{gleichung2}
\|e^{-\alpha \omega}T_s \mathds{1}_B\|_{L^2}\leq e^{\alpha^2
s/2}\|e^{-\alpha \omega}\mathds{1}_B\|_{L^2} \stackrel{\omega =
d(A,B) \, on\, B}{=}e^{\alpha^2 s/2-\alpha d(A,B)}\sqrt{m(B)}.
\end{equation}

Thus we know

\begin{eqnarray*}
\Pp(X_0 \in A,X_s \in B)&=&\|\mathds{1}_A T_s
\mathds{1}_B\|_{L^1}=\|e^{\alpha \omega}\mathds{1}_A e^{-\alpha
\omega}T_s\mathds{1}_B\|_{L^1}\\[1.5ex]
& \leq & \|e^{\alpha \omega}\mathds{1}_A\|_{L^2} \|e^{-\alpha
\omega}T_s\mathds{1}_B\|_{L^2}\\[1.5ex]
&=&\sqrt{m(A)}\sqrt{m(B)} \exp\big[\alpha^2 s/2-\alpha \, d(A,B)
\big].
\end{eqnarray*}

Because this is true for all $\alpha \in \mathds{R}$ we can optimize
in $\alpha$ and obtain for $\alpha = \frac{d(A,B)}{s}$

$$\Pp(X_0 \in A,X_s \in B)\leq \sqrt{m(A)}\sqrt{m(B)} e^{- \frac{d(A,B)}{2\cdot s}}.$$

From this, the claim follows in the case of finite distance.

\vspace{.2cm}

If $d(A,B)=\infty$, we set $\omega = d_A \wedge M$ and obtain

$$\Pp(X_0 \in A,X_s \in B)\leq \sqrt{m(A)}\sqrt{m(B)} e^{- \frac{M^2}{2\cdot s}}.$$

\vspace{.2cm}

For $M \to \infty$ we see $\Prob(X_0 \in A,X_s \in B)=0$ for $t \geq
0$ (cf. \cite{HR} or \cite{SII}).
\end{proof}
\vspace{.2cm}

 After this special case we want to prove the theorem
in a quite similar way. For this purpose we will repeatedly apply
the 'Integrated Gaussian estimates`

\vspace{0.5cm}

\begin{proof}(of Theorem \ref{thm})

For $\mathcal{A}=(A_0,A_1,\ldots,A_n) \in \X^{n+1}$ we want to show

$$\Pp(X_{.}\in \mathcal{A}_\Delta)\leq
\sqrt{m(A_0)}\sqrt{m(A_n)}\exp\big[- \frac{1}{2}\sum_{i=0}^{n-1}
\frac{d^2(A_i,A_{i+1})}{t_{i+1}-t_i}\big].$$
Like in the previous corollary first we assume
$d(A_i,A_{i+1})<\infty$. Define $\omega_i=d_{A_i}\wedge
d(A_i,A_{i+1})$ for $i=0,1,\ldots,n-1$. Then for arbitrary, but fix
$\alpha_i \in \mathds{R},\ i=0,1,\ldots,n-1$,

\begin{eqnarray*}
\lefteqn{ \Pp(X_{.}\in \mathcal{A}_{\Delta})= \|\mathds{1}_{A_0}
    T_{t_1-t_0} \big(\mathds{1}_{A_1} T_{t_2-t_1} \mathds{1}_{A_2}\ldots
        T_{t_{n}-t_{n-1}} \mathds{1}_{A_n}\big)\|_{L^1}{}}\\[2ex]
    &&=\|e^{\alpha_0 \omega_0}\,\mathds{1}_{A_0}
        T_{t_1-t_0}\,e^{-\alpha_0 \omega_0}\big(\mathds{1}_{A_1} T_{t_2-t_1}
        \mathds{1}_{A_2}\ldots
        T_{t_{n}-t_{n-1}} \mathds{1}_{A_n}\big)\|_{L^1}\\[2ex]
    &&\leq  \|\,e^{\alpha_0 \omega_0}\mathds{1}_{A_0}
        T_{t_1}\|_{L^2}\cdot \|e^{-\alpha_0 \omega_0}\, \mathds{1}_{A_1}
        T_{t_2-t_1} \big(
        \mathds{1}_{A_2}\ldots T_{t_{n}-t_{n-1}} \mathds{1}_{A_n}\big)\|_{L^2}\\[1.5ex]
    &&\stackrel{\ref{gleichung1}}{\leq} e^{\alpha_0 \frac{t_1}{2}}\sqrt{m(A_0)}\cdot \|e^{-\alpha_0
        \omega_0}\, e^{\alpha_1 \omega_1}\, \mathds{1}_{A_1} \,e^{-\alpha_1
        \omega_1}\, T_{t_2-t_1}\big(
        \mathds{1}_{A_2}\ldots T_{t_{n}-t_{n-1}} \mathds{1}_{A_n}\big)\|_{L^2}\\[1.5ex]
    &&\leq e^{\alpha_0^2 \frac{t_1}{2}}\sqrt{m(A_0)}\cdot \|e^{-\alpha_0
        \omega_0}\|_{L^{\infty}} \cdot \|\mathds{1}_{A_1} \,e^{\alpha_1
        \omega_1}\|_{L^{\infty}}\\[1.5ex]
    &&\qquad \cdot \ \|e^{-\alpha_1 \omega_1}\,T_{t_2-t_1}\big(
        \mathds{1}_{A_2}T_{t_3-t_2} \mathds{1}_{A_3}\ldots T_{t_{n}-t_{n-1}}
        \mathds{1}_{A_n}\big)\|_{L^2}\\[1.5ex]
    &&\leq \sqrt{m(A_0)} \exp\big[ \alpha_0^2 \frac{t_1}{2} - \alpha_0 d(A_0,A_1)\big]
        \cdot \exp\big[\alpha_1^2 \frac{t_2-t_1}{2}\big]\\[1.5ex]
    &&\qquad \cdot \ \|e^{-\alpha_1 \omega_1}\,
        \mathds{1}_{A_2}T_{t_3-t_2}\big( \mathds{1}_{A_3}\ldots T_{t_{n}-t_{n-1}}
        \mathds{1}_{A_n}\big)\|_{L^2}\\
    &&\leq\ldots\\
    &&\leq \sqrt{m(A_0)}\exp\big[ \alpha_0^2 \frac{t_1}{2} - \alpha_0
        d(A_0,A_1)\big]\cdot \ldots \cdot \exp\big[ \alpha_{n-2}^2
        \frac{t_{n-1}-t_{n-2}}{2} - \alpha_{n-2} d(A_{n-1},A_{n-2})\big]\\[1.5ex]
    && \qquad \cdot \exp\big[\alpha_{n-1}^2 \frac{t_n-t_{n-1}}{2}\big]
        \cdot \|e^{-\alpha_{n-1}\omega_{n-1}}\mathds{1}_{A_n}\|_{L^2}\\[1.5ex]
    &&\stackrel{\ref{gleichung2}}{\leq} \sqrt{m(A_0)}\sqrt{m(A_n)} \exp\Big[ \sum_{i=0}^{n-1}
        \alpha_i^ \frac{t_{i+1}-t_i}{2} - \alpha_i\, d(A_i,A_{i+1})\Big]
\end{eqnarray*}
As in the proof of corollary \ref{korollar} we minimise all
$\alpha_0,\ldots,\alpha_{n-1}$, and thus get the desired estimate
$$\Pp(X_{.}\in \mathcal{A}_{\Delta})\leq \sqrt{m(A_0)}\sqrt{m(A_n)}\exp\Big[-\frac{1}{2}
\sum_{i=0}^{n-1}\frac{d^2(A_i,A_{i+1})}{t_{i+1}-t_i}\Big].$$
\end{proof}

\section{Short--time Behaviour Controlled by the Energy of Curves}

Once again we recall some notation. $(X_t)_{t \geq 0}$ denotes the
Markov process on the probability space $(\Omega, \Pp)$  associated
with the Dirichlet form $\E$ and $X^s_t=X_{s\cdot t},\ t \in [0,1]$,
$s > 0$, the time--scaled process. For simplicity we assume
$\Omega=\mathcal{C}([0,1],\mathcal{X})$ and $X_t(\omega)=\omega(t)$.
Let $\Pp^s$ be the probability measure defined by $\Pp^s(X_t \in
\cdot)=\Pp(X_{s\cdot t} \in \cdot)$. If $\Delta=\{0=t_0 < t_1 <
\ldots < t_{n-1}< t_n=1\}$ is a partition of the unit interval
$[0,1]$, then we denote by $\Pi_{\Delta}$ the projection of a
function to their values at $t_0,t_1,\ldots,t_n$. As before for a
$(n+1)$--tuple $\A=(A_0,A_1,\ldots,A_{n+1}) \in \X^{n+1}$ we set

$$\A_\Delta=\Pi_\Delta^{-1}(\A)=\Big\{\gamma : [0,1] \to \X \mbox{ continuous } \mid \Pi_\Delta(\gamma) \in \A
\Big\}.$$

In this section we assume $\X$ to be pre--compact.

\subsection{Short--time Behaviour of Finite Dimensional Distribution}

First of all we want to control the short--time behaviour of the
finite dimensional distributions of $X_t$ by a discretization of the
energy functional. In our framework this energy functional is
defined as follows.

\begin{defn}
Let $\gamma : [0,t] \to \X$ be a continuous curve in $\X$. Then for
a partition $\Delta$ the discretized energy functional $H_{\Delta}$
of $\gamma$ is defined by

\begin{equation} \label{hdelta}
H_{\Delta}(\gamma)=\frac{1}{2}\sum_{i=0}^{n-1}\frac{d^2(\gamma(t_i),\gamma(t_{i+1}))}{t_{i+1}-t_{i}}.
\end{equation}
\end{defn}

Now we are able to formulate the main theorem of this subsection.

\begin{thm} \label{5.2}
Under the above conditions, we get
\begin{equation} \label{diskreteAbsch}
\limsup_{s \to 0}\,s \, \log \Pp^s(X_{\cdot} \in
\mathcal{A}_{\Delta}) \leq  -\inf_{\gamma \in
\mathcal{A}_{\Delta}}H_{\Delta}(\gamma).
\end{equation}

\end{thm}

\begin{proof}

By theorem (\ref{thm}) we have

\begin{equation}\label{epsilonabschaetzung}
\Pp^s(X_{.}\in \mathcal{A}_{\Delta}) \leq
\sqrt{m(A_0)}\sqrt{m(A_n)}\exp\left(-\frac{1}{2
s}\sum_{i=0}^{n-1}\frac{d^2(A_i,A_{i+1})}{t_{i+1}-t_i}\right).
\end{equation}

\vspace{.5cm}

\textbf{Step 1:}

For every $\varepsilon > 0$ we want to show the following estimate
for the subset $\mathcal{A}=(A,B,C) \in \X^3$

$$\limsup_{s \to 0}\,s \, \log \Pp^s(X_{\cdot} \in \A_\Delta) \leq - \inf_{\gamma \in
\mathcal{A}_{\Delta}}H_{\Delta}(\gamma)+\varepsilon.$$

\vspace{.2cm}

For this we first assume $\diam(B)\leq \delta=\delta(\varepsilon)$
($\diam(B)=\sup \{d(x,y)\,:\, x,y \in B\}$). Without loss of
generality we may assume $d(A,B)<\infty$ and $d(B,C)<\infty$,
otherwise with theorem (\ref{thm}) we would get
$$\Pp^s(X_{\cdot} \in \A_\Delta)\leq \sqrt{m(A)}\sqrt{m(C)}\exp \left(-\frac{1}{s}\big(d(A,B)+d(B,C)\big)\right)=0$$

and hence apparently (\ref{diskreteAbsch}).

For $m$--a.e. $b^* \in B$ we have
\begin{itemize}
\item $d(A,B)\geq d(A,b^*) -\delta$ \quad and \quad
      $d^2(A,B)\geq d^2(A,b^*) -2\delta d(A,b^*)$

\item $d(B,C)\geq d(b^*,C) -\delta$ \quad and \quad
      $d^2(B,C)\geq d^2(b^*,C) -2\delta d(b^*,C)$
\end{itemize}

and hence

\begin{eqnarray*}
\frac{d^2(A,B)}{t_1-t_0} + \frac{d^2(B,C)}{t_2-t_1} &\geq&  \frac{d^2(A,b^*)}{t_1-t_0} + \frac{d^2(b^*,C)}{t_2-t_1}
- 2\delta
\left[\frac{d(A,b^*)}{t_1-t_0}+\frac{d(b^*,C)}{t_2-t_1}\right]\\[1.5ex]
&\geq& \frac{d^2(A,b^*)}{t_1-t_0} + \frac{d^2(b^*,C)}{t_2-t_1} -
2\delta \left[\frac{d(A,B)}{t_1-t_0}+\frac{d(B,C)}{t_2-t_1} +
\frac{\delta(t_2-t_0)}{(t_1-t_0)(t_2-t_0)}\right]
\end{eqnarray*}

\vspace{.2cm}

Since this is true for $m$--a.e. $b^* \in B$ and $d(A,B)$ as well as
$d(B,C)$ are finite, we can choose $\delta = \delta(\varepsilon)$
appropriately to get the expected connection between the discretized
energy functional and the finite dimensional distributions:

\begin{eqnarray}\label{PABC}
\Pp^s(X_{.}\in \mathcal{A}_{\Delta}) &\leq& \sqrt{m(A)}\sqrt{m(C)}
    \exp\left(- \frac{1}{2 s} \left[\frac{d^2(A,B)}{t_1-t_0} +
    \frac{d^2(B,C)}{t_2-t_1}\right]\right)\\[1.5ex] \nonumber
&\leq & \sqrt{m(A)}\sqrt{m(C)} \exp\left( - \frac{1}{2 s} \essi_{b \in B}
    \Big\{\frac{d^2(A,b)}{t_1-t_0} + \frac{d^2(b,C)}{t_2-t_1}\Big\} +
    \varepsilon\right)\\[1.5ex] \nonumber
&\leq & \sqrt{m(A)}\sqrt{m(C)}\exp\left(-\frac{1}{s} \inf_{\gamma \in
    \mathcal{A}_{\Delta}}H_{\Delta}(\gamma)\right)
    \exp\left(\frac{\varepsilon}{s}\right) \label{Fall1a}
\end{eqnarray}

respectively

\begin{equation}\label{logPABC}
\lim_{s \to 0}s \log \Pp^s(X_{.}\in \mathcal{A}_{\Delta})\leq
-\inf_{\gamma \in \mathcal{A}_{\Delta}}H_{\Delta}(\gamma) +
\varepsilon.
\end{equation}

\vspace{.5cm}

From now on let $B$ be pre--compact, i.e. for all $\delta > 0$ there
exists a natural number $N$ and a family
$B^{(1)},B^{(2)},\ldots,B^{(N)}$ of subsets of $\X$, such that $B$
is contained in the union of the family and such that $\diam
(B^{(i)})\leq \delta$ holds for each $B^{(i)}$ in the family.
Therefore we get the estimates (\ref{PABC}) and (\ref{logPABC}),
respectively, for each $i\in \{1,2,\ldots,N\}$ and for $m$--a.e.
$b^{(i)} \in B^{(i)}$.

\begin{eqnarray*}
\Pp^s(X_{.}\in \mathcal{A}_{\Delta}) &\leq & \sum_{i=1}^{N}
    \Pp(X_{s\cdot t_0} \in A,X_{s\cdot t_1} \in B^{(i)},X_{s \cdot t_2}
    \in C)\\[1.5ex]
& \leq & \sqrt{m(A)} \sqrt{m(C)} \sum_{i=1}^N
    \exp\left(-\frac{1}{2s} \Big(\frac{d^2(A,B^{(i)})}{t_1-t_0} +
    \frac{d^2(B^{(i)},C)}{t_2-t_1}\Big)\right)\\[1.5ex]
& \leq & \sqrt{m(A)} \sqrt{m(C)} N \exp\left(-\frac{1}{2s} \inf_{i
    \in \{1,2,\ldots,N\}}\Big\{\frac{d^2(A,B^{(i)})}{t_1-t_0} +
    \frac{d^2(B^{(i)},C)}{t_2-t_1}\Big\}\right)\\[1.5ex]
& \leq & \sqrt{m(A)} \sqrt{m(C)} N \exp\left(-\frac{1}{s}
    \inf_{\gamma \in
    \mathcal{A}_{\Delta}}H_{\Delta}(\gamma)\right)\\[1.5ex]
& & \cdot \exp \left(  \frac{\delta}{2s} \max_{i \in
    \{1,2,\ldots,N\}} \esss_{b^{(i)} \in B^{(i)}} \Big\{\frac{d(A,b^{(i)})}{t_1-t_0} +
    \frac{d(b^{(i)},C)}{t_2-t_1}\Big\}\right)\\[1.5ex]
& \leq & \sqrt{m(A)} \sqrt{m(C)} N \exp\left(-\frac{1}{s}
    \inf_{\gamma \in
    \mathcal{A}_{\Delta}}H_{\Delta}(\gamma)\right)\\[1.5ex]
& & \cdot \exp \left( \frac{\delta}{2s}
    \Big(\frac{d(A,B)}{t_1-t_0}+\frac{d(B,C)}{t_2-t_1}+\diam(B)\frac{t_2-t_0}{(t_1-t_0)(t_2-t_1)}\Big)\right).
\end{eqnarray*}

\vspace{.2cm}

As before we can choose $\delta=\delta(\varepsilon)$ appropriately
to get

$$\lim_{s\to 0}s \log \Pp^s(X_{.}\in \mathcal{A}_{\Delta}) \leq -\inf_{\gamma \in \mathcal{A}_{\Delta}}H_{\Delta}(\gamma) +\varepsilon.$$

\vspace{.2cm}

\begin{rem}
This is the first time we need some more assumptions on the
underlaying space $\X$, namely that all subsets of $\X$ are
pre--compact.
\end{rem}

\textbf{Step 2:}

Now we are looking at the $n+1$--tuple
$\mathcal{A}=(A_0,A_1,\ldots,A_{n+1})\in \X^{n+1}$ assuming that all
the distances $d(A_i,A_{i+1})$, $i \in \{0,1,\ldots,n\}$, are finite
and, as before, $\diam(A_i) \leq \delta$, for $i \in
\{1,2,\ldots,n-1\}$ and arbitrary $\delta
> 0$. Let $\Delta=\{0=t_0<t_1<t_2<\ldots <t_n=1\}$ be a partition of the unit
interval $[0,1]$. Then for $m$--a.e. $a_i \in A_i$ the following
three types of estimates are satisfied:

\begin{itemize}

\item[(i)] $d(A_0,A_1) \geq d(A_0,a_1) -\delta$ \quad and
\quad $d^2(A_0,A_1) \geq d^2(A_0,a_1) -2\delta d(A_0,a_1)$

\item[(ii)] $d(A_{n},A_{n+1}) \geq d(a_{n},A_{n+1}) -\delta$ \quad and
\quad $d^2(A_n,A_{n+1}) \geq d^2(a_n,A_{n+1}) -2\delta
d(a_{n},A_{n+1})$

\item[(iii)] $d(A_i,A_{i+1}) \geq d(a_i,a_{i+1}) -2\delta$ \quad and
\quad $d^2(A_i,A_{i+1}) \geq d^2(a_i,a_{i+1}) -4\delta
d(a_i,a_{i+1})\\ \forall i=1,\ldots,n-1.$
\end{itemize}

So we get the estimate

\begin{eqnarray*}
\lefteqn{\frac{d^2(A_0,A_1)}{t_1-t_0} +
    \sum_{i=1}^{n-1}\frac{d^2(A_i,A_{i+1})}{t_{i+1}-t_i} +
    \frac{d^2(A_n,A_{n+1})}{t_{n+1}-t_n} }\\[1.5ex]
& \geq &\min_{a_i \in A_i,\ i\in
    \{1,2,\ldots,n\}}\left\{\frac{d^2(A_0,a_1)}{t_1-t_0}+
    \sum_{i=1}^{n-1} \frac{d^2(a_i,a_{i+1})}{t_{i+1}-t_i} + \frac{d^2(a_n,A_{n+1})}{t_{n+1}-t_n}\right.\\[1.5ex]
& & \hspace{3cm} \left. - 2\delta
    \left(\frac{d(A_0,a_1)}{t_1-t_0}+ \sum_{i=1}^{n-1} \frac{2d(a_i,a_{i+1})}{t_{i+1}-t_i} +
    \frac{d(a_n,A_{n+1})}{t_{n+1}-t_n}\right)\right\}\\[1.5ex]
& \geq &\min_{a_i \in A_i,\ i\in
    \{1,2,\ldots,n\}}\left\{\frac{d^2(A_0,a_1)}{t_1-t_0}+
    \sum_{i=1}^{n-1} \frac{d^2(a_i,a_{i+1})}{t_{i+1}-t_i} +
    \frac{d^2(a_n,A_{n+1})}{t_{n+1}-t_n}\right\}\\[1.5ex]
&&  \hspace{3cm} -2\delta \left(\frac{d(A_0,A_1)}{t_1-t_0}+
    \sum_{i=1}^{n-1} \frac{2d(A_i,A_{i+1})}{t_{i+1}-t_i} + \frac{d(A_n,A_{n+1})}{t_{n+1}-t_n}\right)\\[1.5ex]
&&  \hspace{3cm} - 2\delta^2
    \left(\frac{1}{t_1-t_0}+\sum_{i=1}^{n-1}
    \frac{4}{t_i-t_{i+1}}+\frac{1}{t_{n+1}-t_n}\right).
\end{eqnarray*}

Since the distances $d(A_i,A_{i+1})$ are all finite by assumption,
$\delta=\delta(\varepsilon)$ can be chosen appropriately in
dependency on $\varepsilon$ such that:

\begin{eqnarray} \label{beschDIAM}
\lefteqn{\Pp^s(X_{\cdot}\in \mathcal{A}_{\Delta})=\Pp(X_{s \cdot
        t_0} \in A_0,X_{s \cdot t_1}\in A_1,\ldots,X_{s \cdot t_{n+1}}\in
        A_{n+1})}\\[1.5ex]\nonumber
& \leq & \sqrt{m(A_0)} \sqrt{m(A_{n+1})} \exp\left[- \frac{1}{2s}
    \left(\frac{d^2(A_0,A_1)}{t_1-t_0} + \sum_{i=1}^{n-1} \frac{d^2(A_i,A_{i+1})}{t_{i+1}-t_i} +
    \frac{d^2(A_n,A_{n+1})}{t_{n+1}-t_n}\right) \right] \\[1.5ex]
    \nonumber
&\leq& \sqrt{m(A_0)} \sqrt{m(A_{n+1})} \exp\left[- \frac{1}{s}
    \inf_{\gamma \in \mathcal{A}_{\Delta}}H_{\Delta}(\gamma) +
    \frac{\varepsilon}{s}\right]
\end{eqnarray}

and hence

\begin{equation*} \label{logviermengen}
\lim_{s \to 0} s \log \Pp^s(X_{\cdot}\in \mathcal{A}_{\Delta}) \leq
-\inf_{\gamma \in \mathcal{A}_{\Delta}}H_{\Delta}(\gamma) +
\varepsilon.
\end{equation*}

To relax the assumption on $A_i$, $i \in \{1,2,\ldots,n\}$, that all
of the subsets $A_i \subset \X$ have to be of diameter smaller then
$\delta$, let each $A_i$ be pre--compact. That is for all $\delta >
0$ there exist natural numbers $N_1,N_2,\ldots,N_n$ and families
$B_i^{(j)}$ of subsets of $\X,\ i \in \{1,2,\ldots,n\}, j\in
\{1,2,\ldots,N_i\}$, such that for each $i$ the family $B_i^{(j)},\
j\in \{1,2,\ldots,N_i\}$, is a finite cover of $B_i$

With the previous estimates we have

\begin{eqnarray*}
\lefteqn{\Pp^s(X_{\cdot}\in \mathcal{A}_{\Delta})}\\[1.5ex]
    &\leq&\sum_{j_i=1}^{N_1}\sum_{j_2=1}^{N_2}\ldots\sum_{j_n=1}^{N_{n}}
    \Pp^s(X_{t_0} \in A_0,X_{t_1} \in A_1^{(j_1)},X_{t_2} \in A_2^{(j_2)},\ldots,X_{t_{n}} \in A_{n}^{(j_n)},X_{t_{n+1}} \in
    A_{n+1})\\[1.5ex]
&\leq& \sqrt{m(A_0)} \sqrt{m(A_{n+1})}
    \sum_{j_i=1}^{N_1}\sum_{j_2=1}^{N_2}\ldots\sum_{j_n=1}^{N_{n}} \exp\left[-\frac{1}{2s}
    \left(
    \frac{d^2(A_0,A_1^{(j_1)})}{t_1-t_0} + \right.\right.\\[1.5ex]
&&  \hspace{7.5cm}\left.\left.+\sum_{i=1}^{n-1}
    \frac{d^2(A_i^{(j_i)},A_{i+1}^{(j_{i+1})})}{t_{i+1}-t_1} +
    \frac{d^2(A_{n}^{(i_n)},A_{n+1})}{t_{n+1}-t_{n}}\right)\right]\\[1.5ex]
&\leq& \sqrt{m(A_0)} \sqrt{m(A_{n+1})}\ N_1 \cdot N_2 \cdot \ldots
    \cdot N_{n}\\[1.5ex]
&&\cdot \exp\left[-\frac{1}{2s}
    \min_{A_1^{(j_1)},A_2^{(j_2)},\ldots,A_{n}^{(j_n)}}\left\{
    \frac{d^2(A_0,A_1^{(j_1)})}{t_1-t_0}
    +\sum_{i=1}^{n-1} \frac{d^2(A_i^{(j_i)},A_{i+1}^{(j_{i+1})})}{t_{i+1}-t_1} +
    \frac{d^2(A_{n}^{(i_n)},A_{n+1})}{t_{n+1}-t_{n}}\right\}\right]\\[1.5ex]
&\leq&  \sqrt{m(A_0)} \sqrt{m(A_{n+1})}\ N_1 \cdot N_2 \cdot \ldots
    \cdot N_{n} \cdot \exp\left[-\frac{1}{s}\inf_{\gamma \in
    \mathcal{A}_{\Delta}}H_{\Delta}(\gamma)\right]\\[1.5ex]
&&\cdot \exp\left[\max_{A_1^{(j_1)},A_2^{(j_2)},\ldots,A_{n}^{(j_n)}}\left\{\frac{\delta}{s} \Big(\frac{d(A_0,A_1)}{t_1-t_0}+
    \sum_{i=1}^{n-1}\frac{2d(A_i,A_{i+1})}{t_{i+1}-t_i} + \frac{d(A_{n},A_{n+1})}{t_{n+1}-t_{n}} \Big)\right\}\right]\\[1.5ex]
&&\cdot \exp\left[ \frac{\delta^2}{s}
        \Big(\frac{1}{t_1-t_0}+\sum_{i=1}^{n-1}\frac{4}{t_{i+1}-t_i}+\frac{1}{t_{n+1}-t_{n}}\Big)\right]\\[1.5ex]
&\leq&  \sqrt{m(A_0)} \sqrt{m(A_{n+1})}\ N_1 \cdot N_2 \cdot \ldots
    \cdot N_{n} \cdot \exp\left[-\frac{1}{s}\inf_{\gamma \in
    \mathcal{A}_{\Delta}}H_{\Delta}(\gamma)\right]\\[1.5ex]
&&\cdot \exp\left[\frac{\delta}{s} \Big(\frac{d(A_0,A_1)}{t_1-t_0}+
    \sum_{i=1}^{n-1}\frac{2d(A_i,A_{i+1})}{t_{i+1}-t_i} + \frac{d(A_{n},A_{n+1})}{t_{n+1}-t_{n}} +3\sum_{i=1}^{n-1}
    \diam(A_i)\Big)\right]\\[1.5ex]
&&\cdot \exp\left[ \frac{\delta^2}{s}
        \Big(\frac{1}{t_1-t_0}+\sum_{i=1}^{n-1}\frac{4}{t_{i+1}-t_i}+\frac{1}{t_{n+1}-t_{n}}\Big)\right]\\[1.5ex]
&\leq& \sqrt{m(A_0)}\sqrt{m(A_{n+1})}\  N_1 \cdot N_2 \cdot \ldots
    \cdot N_{n} \cdot \exp\left[- \frac{1}{s} \inf_{\gamma \in
    \mathcal{A}_{\Delta}}H_{\Delta}(\gamma) +
    \frac{\varepsilon}{s}\right],
\end{eqnarray*}

after a appropriate choice of $\delta(\varepsilon)$. (We can choose
$\delta$ in such a way, because all subsets $A_i$ for $i\in
\{1,2,\ldots,n\}$ are pre--compact and hence $\diam(A_i)$ is
finite.) Here $\min_{A_1^{(j_1)},A_2^{(j_2)},\ldots,A_{n}^{(j_n)}}$
means that we minimize for each $i$ over all possible $A_i^{(j_i)}$,
$j \in \{1,2,\ldots,N_i\}$. Hence for each $\varepsilon >0$ the
following inequality holds:

\begin{equation} \label{fast}
\lim_{s \to 0} s \log \Pp^s(X_{\cdot}\in \mathcal{A}_{\Delta})
\leq -\inf_{\gamma \in \mathcal{A}_{\Delta}} H_{\Delta}(\gamma) +
\varepsilon.
\end{equation}

As the estimate \ref{fast} is true for all $\varepsilon > 0$ the
theorem \ref{thm} is proven.
\end{proof}

\subsection{Short--time behaviour of the Markov process}

Until now we described the asymptotic short--time behaviour of the
finite dimensional distributions of the Markov process $X_t$
associated with the Dirichlet form $\E$ via the discretized Energy
functional $H_\Delta$. In the following we want to lift up this
results to an estimate of the short--time behaviour of the law of
$X_t$ itself. For this we will apply a version of the theorem of
Dawson--G\"artner. This yields the weak Large Deviation Principle
(LDP) in a space $\Y$ as a consequence of the LDP's in $\Y_i$, where
$\Y$ is the projective limit of the projective system $\Y_i$.

To formulate the theorem of Dawson--G\"artner precisely we have to
recall some well known concepts. We mention that a LDP describes the
asymptotic behaviour, as $\varepsilon \to \infty$, of a family of
probability measures $\{\mu_\varepsilon\}$ on $(\Omega,\mathcal{B})$
in terms of a rate function, where a rate function is defined as
follows.

\begin{defn}
A function $I : \Omega \to [0,\infty]$ is called a rate function if
it is lower semi--continuous.

We say that a function $I : \Omega \to [0,\infty]$ is a good rate
function, if $I$ is lower semi--continuous and for all $\alpha \in
[0,\infty)$ the level sets $\psi_I(\alpha)=\{x \in \Omega : I(x)\leq \alpha\}$ are
compact subsets of $\Omega$.
\end{defn}

For any set $\Gamma$, $\overline{\Gamma}$ denotes the closure of
$\Gamma$ and $\Gamma^\circ$ the interior of $\Gamma$. Then we say

\begin{defn}
The family $\{\mu_\varepsilon\}$ of probability measures satisfies
the LDP with good rate function $I$ if, for all subsets $\Gamma \in
\mathcal{B}$,

$$-\inf_{\omega \in \Gamma^\circ}I(\omega) \leq \liminf_{\varepsilon \to 0}
\varepsilon \log \mu_\varepsilon (\Gamma)
\leq \limsup_{\varepsilon \to 0} \varepsilon \log \mu_\varepsilon
(\Gamma) \leq -\inf_{\omega \in \overline{\Gamma}}I(\omega).$$

The infimum of a function over an empty set is interpreted as
$\infty$.
\end{defn}

There is an other weaker form of a LDP where the upper bound is proven only for compact sets.

\begin{defn}
A family of probability measures $\{\mu_\varepsilon\}$ is said to satisfy the weak LDP with rate function $I$ if the upper bound
\begin{equation}
\limsup_{\varepsilon \to 0} \varepsilon \log \mu_{\varepsilon}(\Gamma)\leq -\alpha
\end{equation}
holds for all $\alpha < \infty$ and all compact subsets $\Gamma$ of the complement of level sets $\psi_I(\alpha)^C$ and the lower bound
\begin{equation}
\liminf_{\varepsilon \to 0} \varepsilon \log \mu_{\varepsilon}(\Gamma)\geq -I(x)
\end{equation}
holds for any $x\in \{y : I(y)<\infty\}$ and all measurable $\Gamma$ with $x \in \Gamma^{\circ}$.
\end{defn}

Let $J$ be a partial ordered set and $\{(\Y_j,p_{ij})\}_{i\leq j \in
\N}$ be a projective system, i.e. $\{\Y_j\}_{j \in J}$ is a family
of Hausdorff topological spaces and the continuous maps $p_{ij} :
\Y_j \to \Y_i$ satisfy $p_{ik}=p_{ij}\circ p_{jk}$ for all $i\leq
j\leq k$. Let $\Y = \lim\limits_{\longleftarrow}\Y_j$ be the
projective limit of this system, that is $\Y$ consists of all the
elements $\mathbf{y}=(y_j)_{j\in J}$ for which $y_i=p_{ij}(y_j)$
whenever $i<j$. Then the statement of the theorem of
Dawson--G\"artner reads as

\begin{thm} (Dawson--G\"artner) (cf. \cite{dembo}) \label{dawsongaertner}\\
Let $\{\mu_\varepsilon\}$ be a family of probability measures on
$\Y$. Assume that, for each $j\in J$, the family of push--forward
measures $\{{p_j}_* \mu_\varepsilon\}$ on $\Y_j$ satisfy the LDP
with good rate function $I_j : \Y_j \to [0,\infty]$. Then the family
$\{\mu_\varepsilon\}$ satisfies the LDP on $\Y$ with good rate
function $I : \Y \to [0,\infty]$ given by

$$I(\mathbf{y})=\sup_{j \in J}\{I_j(p_j(\mathbf{y}))\},\quad \mathbf{y} \in \Y.$$

\end{thm}

\begin{rem}
For the lower bound it is not necessary to assume the functional $I$ to be a good rate function, i.e. we do not have to assume that all the level sets are compact. On the other hand for the upper bound it is crucial assumption that they are all compact.
\end{rem}

To abolish having not a good rate function we can formulate the following corollary

\begin{cor} \label{weakdawsongaertner}
Let $\{\mu_\varepsilon\}$ be a family of probability measures on
$\Y$. Assume that, for each $j\in J$, the family of push--forward
measures $\{{p_j}_* \mu_\varepsilon\}$ on $\Y_j$ satisfy the weak LDP
with rate function $I_j : \Y_j \to [0,\infty]$. Then the family
$\{\mu_\varepsilon\}$ satisfies the weak LDP on $\Y$ with rate
function $I : \Y \to [0,\infty]$ given by

$$I(\mathbf{y})=\sup_{j \in J}\{I_j(p_j(\mathbf{y}))\},\quad \mathbf{y} \in \Y.$$

\end{cor}

\begin{proof}
The proof works most like the proof of the theorem (\ref{dawsongaertner}) of Dawson and G\"artner, for the lower bound it is exactly the same. For the upper bound first we get $\psi_{I_i}(\alpha)=p_{ij}\left(\psi_{I_i}(\alpha)\right)$ for all $i<j$ because all of the level sets $\psi_{I_j}(\alpha)$ of $I_j$ are closed subsets of $\mathcal{Y}_j$. Hence we get
$$\psi_I(\alpha)= \lim_{\longleftarrow} \psi_{I_j}(\alpha),$$
and $\psi_I(\alpha)$ as the projective limit of closed sets is itself a closed subset of $\mathcal{Y}$.

Now we take a compact subset $\Gamma \subset \mathcal{Y}$ and consider the projections $\Gamma_j:=p_j(\Gamma)$, since $p_j: \mathcal{Y}\to \mathcal{Y}_j$ is continuous this sets are also compact and we get
$$\Gamma=\lim_{\longleftarrow} \Gamma_j$$
and consequently
$$\Gamma \cap \psi_I(\alpha)=\lim_{\longleftarrow} \left( \Gamma_j \cap \psi_{I_j}(\alpha)\right).$$
For all $\alpha > 0$ and all compact subsets $\Gamma$ of $\psi_{I}(\alpha)^C$ (i.e. $\Gamma \cap \psi_I(\alpha)=\emptyset$) we have $\Gamma_j \cap \psi_{I_j}(\alpha)=\emptyset$ for some $j\in J$ (cf. theorem B.4 in (\cite{dembo})). Thus we get
$$\limsup_{\varepsilon \to 0} \varepsilon \log \mu_{\varepsilon}(\Gamma) \leq \limsup_{\varepsilon \to 0} \varepsilon \log \mu_{\varepsilon} \circ p_j^{-1}(\Gamma_j)\leq -\alpha.$$
\end{proof}

Now we come back to the situation of the previous sections. $X_t$,
$t \geq 0$, is the Markov process on a probability space
$(\Omega,\Pp)$ associated with the Dirichlet form $(\E,\D)$ where
$\D \subset L^2(\X,m)$. As before, we assume
$\Omega=\mathcal{C}([0,1],\X)$. Then $\Pp^s$ is the distribution of
the time--scaled Markov process $X_{\cdot}^s$ where $X_t^s=X_{s
\cdot t}$ for $t \in [0,1],\ s \geq 0$. That is $\Pp^s(X_t \in
\cdot)=\Pp(X_{s \cdot t} \in \cdot)$. Let

$$J=\bigcup_{n=0}^{\infty}\{\Delta^{n} :
\Delta^{n}=\{0=t_0 < t_1 <\ldots < t_{n}=1\} \mbox{ partition of } [0,1] \},$$

be the union of all partitions $\Delta^n$ of the unit interval
$[0,1]$. A partial ordering on $J$ is induced by inclusion.

It is a well known fact, that the family of finite dimensional
distributions $\{{\Pi_\Delta}_*\mu \}$ of a stochastic process $X$
on a probability space $(\Omega,\A,\mu)$ with values in
$(E,\mathcal{B})$ together with $\Pi_\Delta $ form a projective
system, here $\Pi_\Delta $is the projection of functions onto their
values at the time instances $t_1,t_2,\ldots,t_n$ of $\Delta$.

We have seen in the last section in theorem (\ref{thm}), that the
family of finite dimensional distribution $\{{\Pp^s}_*\Pi_\Delta\}$
of the time--scaled Markov process $X_t^s$ associated to the
Dirichlet form $\E$ satisfies the upper estimate of the weak LDP with
good rate function $H_\Delta$, where $H_\Delta$ is the discretized
energy functional as defined in (\ref{hdelta}). According to the
discussion preceding we can apply corollary \ref{weakdawsongaertner}.
Then we get, that $\{\Pp^s\}$ satisfies the upper estimate of the weak
LDP with good rate function

\begin{equation}\label{energiefunktional}
H(\gamma)=\sup_{\Delta \in J}H_\Delta(\gamma),
\end{equation}

for a continuous function $\gamma : [0,1] \to \X$.

\begin{rem}
If we take a refinement $\widetilde{\Delta}=\{0=t_0<t_1<\ldots <
t_i<r<t_{i+1}<\ldots < t_n=1\}$ of a partition
$\Delta=\{0=t_0<t_1<\ldots < t_i<t_{i+1}< \ldots < t_n=1\}$ of the
unit interval, it is easy to see $
H_{\Delta}<H_{\widetilde{\Delta}}$, because

\begin{eqnarray*}
\frac{d^2(\gamma_{t_i},\gamma_{t_{i+1}})}{t_{i+1}-t_i} &\leq&
    \frac{\big[d(\gamma_{t_i},\gamma_r)+d(\gamma_r,\gamma_{t_{i+1}})\big]^2}{t_{i+1}-t_i}\\
&\leq& \frac{1}{t_{i+1}-t_i}\left[\frac{t_{i+1}-t_i}{r-t_i}d^2(\gamma_{t_i}-\gamma_r)
    +\frac{t_{i+1}-t_i}{t_{i+1}-r}d^2(\gamma_r-\gamma_{t_{i+1}})\right] \\
&=& \frac{d^2(\gamma_{t_i},\gamma_r)}{r-t_i}+\frac{d^2(\gamma_r,\gamma_{t_{i+1}})}{t_{i+1}-r}.
\end{eqnarray*}
We use $(a+b)^2\leq(1+\lambda)a^2+(1+\frac{1}{\lambda})b^2$.
\end{rem}
In the following we want to get a more explicit expression for the energy $H$. For this we consider absolutely continuous curves $\gamma \in AC^2([0,1],\X)$ with finite $2$-energy. This are curves for which exists $m \in L^2([0,1])$ such that
\begin{equation}\label{ac2}
d(\gamma(s),\gamma(t))\leq \int_s^t m(r)\mathrm{d}r \quad \forall s,t \in [0,1],\ s\leq t.
\end{equation}
This curves have the property to be differentiable (in the metric sense) a.e.. To be more precise the following theorem (cf. \cite{Ambrosio}) holds

\begin{thm}\label{metrabl}
Let $\gamma \in AC^2([0,1],\X)$. Then for Lebesgue-a.e. $t\in [0,1]$ there exists the limit
\begin{equation}
|\dot{\gamma}|(t):= \lim_{h \to 0} \frac{d(\gamma(t),\gamma(t+h))}{|h|}.
\end{equation}
Furthermore $|\dot{\gamma}|\in L^2$ and we know $d(\gamma(s),\gamma(t))\leq \int_s^t |\dot{\gamma}|(r)\mathrm{d}r$. Moreover $|\dot{\gamma}|(t)\leq m(t)$ for Lebesgue-a.e. $t \in [0,1]$, for all $m$ such that (\ref{ac2}) holds.
\end{thm}

Now we are able to formulate following lemma

\begin{lemma}\label{acenergy}
For all $\gamma \in \Omega$ we define
\begin{equation}
\widetilde{H}(\gamma):=\left\{ \begin{array}{cl}
                            \frac{1}{2}\int_0^1|\dot{\gamma}|^2(r)\mathrm{d}r & \mathrm{,\,if } \gamma \in AC^2([0,1],\X)\\
                            \infty & \mathrm{,\,else.}
                            \end{array}\right.
\end{equation}
Then $H(\gamma)\leq \widetilde{H}(\gamma)$.
\end{lemma}

\begin{proof}
(i): $\gamma \not \in AC^2 \quad \Longrightarrow \quad H(\gamma)\leq \widetilde{H}(\gamma)=\infty.$ \checkmark\\[2ex]
(ii): $\gamma \in AC^2$: Let $\Delta^n=\{0=t_0<t_1<\ldots<t_n=1\}$ be an arbitrary partition, then
\begin{eqnarray*}
\frac{1}{2}\sum_{i=0}^{n-1}\frac{d^2(\gamma(t_{i}),\gamma(t_{i+1}))}{t_{i+1}-t_i} &=& \frac{1}{2}\sum_{i=0}^{n-1} (t_{i+1}-t_i) \left(\frac{d(\gamma(t_{i}),\gamma(t_{i+1}))}{t_{i+1}-t_i}\right)^2\\
&\leq& \frac{1}{2} \sum_{i=0}^{n-1} (t_{i+1}-t_i)^{-1} \left(\int_{t_i}^{t_{i+1}} |\dot{\gamma}|(r) \mathrm{d}r\right)^2\\
&\leq&\frac{1}{2}  \sum_{i=0}^{n-1} \int_{t_i}^{t_{i+1}} |\dot{\gamma}|^2(r) \mathrm{d}r = \frac{1}{2} \int_0^1 |\dot{\gamma}|^2(r) \mathrm{d}r=\widetilde{H}(\gamma).
\end{eqnarray*}
Since this holds true for all partitions we get $H(\gamma)\leq \widetilde{H}(\gamma)$.
\end{proof}

The next goal is to prove equality in the conclusion of lemma \ref{acenergy}, namely

\begin{thm}
Let $\gamma \in \Omega$ and $H(\gamma)$ and $\widetilde H(\gamma)$ defined as above. Then
$$H(\gamma)=\widetilde{H}(\gamma).$$
\end{thm}

\begin{proof}
It remains to show $\widetilde{H}(\gamma)\leq H(\gamma)$. First of all we observe that if $\gamma \notin AC^2$ then $\sup_{\Delta^n} \sum_{i=0}^{n-1} d(\gamma(t_i),\gamma(t_{i+1}))=\infty$ and hence also $\sup_{\Delta^n} \sum_{i=0}^{n-1} \frac{d^2(\gamma(t_i),\gamma(t_{i+1}))}{t_{i+1}-t_i}=\infty$. Consequently we know $\widetilde{H}(\gamma)=\infty = H(\gamma)$ for all $\gamma \notin AC^2$.

On the other hand if $\gamma \in AC^2$ we see $\sup_{\Delta^n} \sum_{i=0}^{n-1} d(\gamma(t_i),\gamma(t_{i+1})) \leq \sup_{\Delta^n} \int_0^1 m(r) < \infty$ where $m$ is a $L^2$ function (cf. (\ref{ac2})). So in the following considerations it is adequate only to take care about continuous curves $\gamma$ with finite length.

For such a $\gamma$ we define the discrete measure
$$\nu_N:=\sum_{i=0}^{N-1} d\left(\gamma \left( \frac{i}{N}\right),\gamma \left(\frac{i+1}{N} \right)\right).$$
This bounded monotone sequence converges up to subsequences to a measure $\nu$ for $N \to \infty$. Further we know
$$d(\gamma(s),\gamma(t)) \leq \nu_N([s,t]) \quad \mbox{for all } 0\leq s \leq t \leq1 \mbox{ and } s,t \in \left\{0,\frac{1}{N},\ldots,\frac{N-1}{N},1\right\}.$$
Passing to the limit yields
\begin{equation}\label{dleqnu}
d(\gamma(s),\gamma(t)) \leq \nu([s,t]) \quad \mbox{for all } 0\leq s \leq t \leq1.
\end{equation}

Consider
\begin{equation}
\mathcal{E}(\nu|\mu):=\left\{ \begin{array}{cl}
                            \int_0^1|\frac{\mathrm{d}\,\nu}{\mathrm{d}\mu}|^2\mathrm{d}\,\mu & \mathrm{,if\ } \nu \ll \mu\\
                            \infty & \mathrm{,else.}
                            \end{array}\right.
\end{equation}
This is a joint semicontinuous functional.

Let
$$\mu_N:=\sum_{i=0}^{N-1}\frac{1}{N}\delta_{i/N} \rightharpoonup \mu$$
where $\mu$ is the Lebesgue measure on $[0,1]$. Then

$$\mathcal{E}(\nu_N|\mu_N)=\int_0^1 |\frac{\mathrm{d}\,\nu_N}{\mathrm{d}\,\mu_N}|^2\mathrm{d}\mu_N=\sum_{i=0}^{N-1} \frac{d^2\left(\gamma\left(\frac{i}{N}\right),\gamma\left(\frac{i+1}{N}\right)\right)}{1/N}\leq \sup_{\Delta^n} \sum_{i=0}^{n-1}\frac{d^2(\gamma(t_i),\gamma(t_{i+1}))}{t_{i+1}-t_i}=2H(\gamma).$$

So if $H(\gamma) < \infty$ then also $\mathcal{E}(\nu|\mu)< \infty$ and therefore $\nu$ is absolutely continuous with respect to the Lebesgue measure $\mu$. To be more precise $\nu=f\,\mu$ with $||f||_2\leq \sqrt{2H(\gamma)}.$

Together with (\ref{dleqnu}) we see
$$d(\gamma(s),\gamma(t))\leq \nu([s,t])=\int_s^t f(r)\mathrm{d}r.$$
Then theorem \ref{metrabl} yields $|\dot{\gamma}|_r \leq f(r)$ for Lebesgue-a.e. $r \in [0,1]$. Hence we get for $\gamma \in AC^2$

$$\int_0^1|\dot{\gamma}|^2_r \mathrm{d}r \leq \int_0^1 f(r)^2 \mathrm{d}r \leq\int_0^1 \sup_{\Delta^n}\sum_{i=0}^{n-1}\frac{d^2(\gamma(t_i),\gamma(t_{i+1}))}{t_{i+1}-t_i}\,\mathrm{d}r=2H(\gamma).$$
\end{proof}

The argument of the last proof was communicated to us by Professor L. Ambrosio.

Now we are able to state our main theorem

\begin{thm} \label{mainthm}
For all $\alpha>0$ and all compact subsets $\Gamma$ of $\{\gamma : \widetilde{H}(\gamma)>\alpha\}$ the following holds

$$\limsup_{s \to 0}\,s \, \log \Pp^s(X_{\cdot} \in \Gamma)
\leq  -\alpha.$$

\end{thm}

\section{Application to Wasserstein Diffusion}

At the end we want to present an application of our work to the
Wasserstein diffusion as it is introduced in \cite{RS07}. The
Wasserstein diffusion can be regarded as a stochastic perturbation
of the heat flow on $\mathcal{P}([0,1])$, the space of probability
measures on the unit interval.

To be more precise in \cite{RS07} it is constructed a probability
measure $\mathbf{P}^\beta$ on $\mathcal{P}([0,1])$ formally given as

$$d \mathbf{P}^\beta (\mu)=\frac{1}{Z_\beta}e^{-\beta \mathrm{Ent}(\mu)}d\mathbf{P}(\mu),$$

here $\mathbf{P}$ is a 'uniform distribution' on
$\mathcal{P}([0,1])$, $\beta > 0$, $Z_\beta$ a normalization
constant and $\mathrm{Ent}$ the relative entropy. One important
result in \cite{RS07} is that the Wasserstein Dirichlet form

$$\Ee(u,u)=\int_{\mathcal{P}}\|D u (\mu)\|^2_{L^2(\mu)}d\mathbf{P}^\beta$$ is a strongly
local, regular, recurrent Dirichlet form on
$L^2(\mathcal{P}([0,1]),\mathbf{P}^\beta)$. Hence there exists a
strong Markov process $((\mu_t)_{t\geq 0},\Pp^\beta)$ on
$\mathcal{P}([0,1])$ associated with the Dirichlet form $\Ee$.
$(\mu_t)_{t \geq 0}$ is called Wasserstein diffusion. It is also
shown in \cite{RS07} that the intrinsic metric for the Dirichlet
form $\Ee$ is the $L^2$--Wasserstein distance $d_W$.

If $\Pp^{\alpha,\beta}$ denotes the law of the rescaled process
$(\mu_{\alpha \cdot t})_{t \in [0,1]},\ \alpha \geq 0$, then theorem
(\ref{mainthm}) reads as

\begin{thm}
For all $\kappa$ and all compact subsets $\Gamma$ of $\{\mu  \in \mathcal{C}([0,1],\mathcal{P}([0,1])) : \widetilde{H}(\mu) >\kappa\}$  we have
\begin{equation*}
\limsup_{\alpha \to 0} \alpha \log \Pp^{\alpha,
\beta}(\mu_{\cdot}\in \Gamma) \leq - \kappa
\end{equation*}
where
\begin{equation*}
\widetilde{H}(\nu)=\left\{
    \begin{array}{cl}
    \frac{1}{2} \int_0^1|\dot{\nu}|^2(t)dt & \mbox{,if } \nu \in AC^2\big((\mathcal{P}([0,1]),d_W)\big)\\
    \infty & \mbox{,else.}
    \end{array}
    \right.
\end{equation*}
\end{thm}

There is an other interesting point of view. If we make use of the
representation of probability measures by their inverse distribution
functions we can regard the Wasserstein diffusion $\mu_t$ as a
process $g_t$ on $\mathcal{G}$, the space of non--decreasing
functions from $[0,1]$ into itself. In particular, the map $\chi :
\mathcal{P}([0,1])\to \mathcal{G}$, which assigns to each
probability measure $\mu_t \in \mathcal{P}([0,1])$ its inverse
distribution function

$$g_{t}(s)=\inf \{a \in [0,1] : \mu_t([0,a])> s\}$$

with $\inf \emptyset =1$, establishes an isometry between
$(\mathcal{P}([0,1]),d_W)$ and $(\mathcal{G},\|\cdot\|_{L^2})$. Here
the $L^2$--distance on $\mathcal{G}$ is defined as usual

$$\|g-h\|_{L^2}=\left(\int_0^1 |g(s)-h(s)|^2ds\right)^{1/2}.$$

Thus we have the equality

$$d_W(\mu_{t_1},\mu_{t_2})=\|g_{t_1}-g_{t_2}\|_{L^2}.$$

As mentioned above via this construction we have a process
$g_t=g_{\mu_t}$ on $\mathcal{G}$ with associated probability measure
$\mathbb{Q}^\beta$. The asymptotic for the time--scaled process
$g_{t \cdot \alpha}$ is then given by the following theorem

\begin{thm}
For all $\kappa$ and all compact subsets $\Gamma$ of $\{g  \in \mathcal{C}([0,1],\mathcal{G}) : \widetilde{H}(g) >\kappa\}$  we have
\begin{equation*}
\limsup_{\alpha \to 0} \alpha \log \mathbb{Q}^{\alpha \beta}(g_{\cdot} \in \Gamma) \leq -\kappa,
\end{equation*}
where
$$\widetilde{H}(h)=\left\{
    \begin{array}{cl}
    \frac{1}{2}\int_0^1 \int_0^1 |\partial_t h_t(s)|^2 dt\,ds & \mbox{,if }h\in AC^2((\mathcal{G},\|\cdot\|_{L^2}))\\
    \infty & \mbox{,else.}
    \end{array}\right.$$
\end{thm}
\vspace{1.5cm}

\bibliographystyle{plain}
\bibliography{lit}

\end{document}